\documentclass[11pt, english]{smfart}
\usepackage[T1]{fontenc}
\usepackage[english,francais]{babel}
\usepackage{graphicx}
\usepackage{amssymb}
\usepackage[new]{old-arrows}
\usepackage{mathtools}
\usepackage{tikz}
\usepackage{tikz-cd} 
\usepackage{amsmath}
\usetikzlibrary{matrix}
\usepackage{smfthm}
\usepackage{mathrsfs}
\usepackage{pb-diagram}
\usepackage{yfonts}
\usepackage[inline]{enumitem}
\usepackage{color}
\usepackage{verbatim}
\usepackage[numbers]{natbib}
\usepackage{dsfont}
\usetikzlibrary{matrix}
\def\R{\mathbb{R}}
\def\Q{\mathbb{Q}}

\def\Z{\mathbb{Z}}
\def\N{\mathbb{N}}
\def\O{\mathcal{O}}
\def\L{\Lambda}

\def\nm{\lVert\cdot\rVert}

\def\vol{\widehat{\mathrm{vol}}}

\def\Div{\mathrm{Div}}
\def\mumin{\widehat{\mu}_{\min}}
\def\mumax{\widehat\mu_{\max}}

\def\limnto{\lim\limits_{n\rightarrow +\infty}}

\def\ot{\otimes}

\def\shfL{\mathscr{L}}
\def\shfM{\mathscr{M}}

\def\ardeg{\widehat{\mathrm{deg}}}

\def\degH{\mathrm{deg}_{\mathcal{H}}}
\def\rank{\mathrm{rk}}
\def\ovl{\overline}

\def\scrX{\mathscr{X}}

\def\xanomega{X^{\mathrm{an}}_\omega}

\usepackage[normalem]{ulem}
\definecolor{mred}{rgb}{0.83, 0.0, 0.0}
\definecolor{darkspringgreen}{rgb}{0.09, 0.45, 0.27}
\definecolor{ruby}{rgb}{0.88, 0.07, 0.37}
\def\colorsout#1{\bgroup\markoverwith{\textcolor{#1}{\rule[0.5ex]{2pt}{0.7pt}}}\ULon} %[0.5ex]{2pt}{0.4pt}
\def\coloruline#1{\bgroup\markoverwith{\textcolor{#1}{\rule[-0.5ex]{2pt}{0.7pt}}}\ULon} %[0.5ex]{2pt}{0.4pt}

\newcommand{\rest}[2]{\left.{#1}\right\vert_{{#2}}}
\newtheorem{claim}{Claim}

\makeatletter
\newcommand\tint{\mathop{\mathpalette\tb@int{t}}\!\int}
\newcommand\bint{\mathop{\mathpalette\tb@int{b}}\!\int}
\newcommand\tb@int[2]{%
  \sbox\z@{$\m@th#1\int$}%
  \if#2t%
    \rlap{\hbox to\wd\z@{%
      \hfil
      \vrule width .35em height \dimexpr\ht\z@+1.4pt\relax depth -\dimexpr\ht\z@+1pt\relax
      \kern.05em % a small correction on the top
    }}
  \else
    \rlap{\hbox to\wd\z@{%
      \vrule width .35em height -\dimexpr\dp\z@+1pt\relax depth \dimexpr\dp\z@+1.4pt\relax
      \hfil
    }}
  \fi
}
\makeatother

\usepackage{etoolbox}
\makeatletter
\newcommand*\suppresschapternumber{%
  \let\@makechapterhead\@makeschapterhead
  \patchcmd{\@chapter}
    {\protect\numberline{\thechapter}}
    {}
    {}{}%
}
\newcommand*\removedotbetweenchapterandsection{%
  \renewcommand\thesection{\thechapter\@arabic\c@section}%
}
\makeatother

\newtheorem{theorem}{Theorem}

\title{Slope boundedness and equidistribution theorem}

\author{Wenbin LUO\thanks{
This work was supported by JSPS KAKENHI Grant Number JP20J20125.}}
\address{Beijing International Center for Mathematical Research, Peking University,
Beijing 100871, China}
\date{\today}

\begin{document}
\begin{abstract}
In this article, we prove the boundedness of minimal slopes of adelic line bundles over function fields of characteristic $0$. This can be applied to prove the equidistribution of generic and small points with respect to a big and semipositive adelic line bundle.  Our methods can be applied to the finite places of number fields as well. We also show the continuity of $\chi$-volumes over function fields.
\end{abstract}
\maketitle
\tableofcontents
\section*{Introduction}
\subsection*{Backgrounds} The equidistribution theorem is an important result for algebraic dynamical systems. It was originally proved by Szpiro, Ullmo and Zhang \citep{SUZ,ullmo1998positivite, zhang1998equidistribution} over number fields in order to prove the Bogomolov conjecture. In \citep{Yuan_2008}, Yuan proved the equiditribution theorem over number fields (including finite places) by using his arithmetic analogue of Siu's inequality. Faber and Gubler \citep{Faber_2009,gubler2008equidistribution} transferred Yuan's result to function fields. In \citep{moriwaki2000arithmetic}, Moriwaki gave a generalization of equidistribution theorem over an arithmetic function field, that is, a finitely generated field over $\Q$.

Here we briefly introduce the application of the equidistribution theorem over function fields. Let $K$ be a function field of a projective normal variety $B$ and $X$ be a projective $K$-variety of dimension $d$. Consider an automorphism $f:X\rightarrow X$ such that there is an ample line bundle $L$ satisfying $f^*L\simeq mL$ for some $m>1$. For each $1$-codimensional point $\omega\in B$, by applying Tate's limit process\citep[9.5]{BombieriGubler2006Diophan}, we can obtain a canonical metric $\phi_\omega$ of $L$ on the Berkovich analytification $\xanomega$ of $X$ with respect to the absolute value given by $\omega$. We say a point $x\in X(\ovl K)$ is \textit{preperiodic} if the orbit $\{f^n(x)\}_{n\in\N}$ is finite. Then the equidistribution theorem over function fields\citep{Faber_2009,gubler2008equidistribution} tells us that the preperiodic points are "equidistributed" on the Berkovich space $\xanomega$ with respect to the Chambert-Loir measure given by $\phi_\omega$. If $K$ is a number field, we encounter the Archimedean places, over which we just use the measure given by the first Chern class of a semi-postive smooth Hermitian metric. Chen and Moriwaki recently proved a result over adelic curves assuming $L$ is semiample \citep{huayimoriwaki_equi}. In this article, we consider the case that $L$ is nef and big. 
\subsection*{Adelic point of view}
We start with a general framework over adelic curves which was introduced by Chen and Moriwaki\citep{adelic}. An \textit{adelic curve} consists of a field $K$, a measure space $(\Omega, \mathcal A, \nu)$ and a set $\{\lvert\cdot\rvert_\omega\}_{\omega\in\Omega}$ of $K$'s absolute values. These structures can be constructed for number fields, function fields and countably generated fields over $\Q$.
In this article, we assume that the adelic curve is proper, that is, it satisfies a product formula (see subsection \ref{sub_adelic} for details). For each $\omega\in\Omega$, we denote by $K_\omega$ the completion of $K$ with respect to $\lvert\cdot\rvert_\omega$. We may further assume that $\Omega$ is discrete or $K$ admits a countable subfield dense in every completion $K_\omega$. 

Let $\pi:X\rightarrow \mathrm{Spec}K$ be a $d$-dimensional, projective, normal and geometrically reduced variety and $L$ be a line bundle over $X$. For each $\omega\in\Omega$, we equip $L$ with a continuous metric $\phi_\omega$ on the Berkovich analytification $\xanomega$ of $X\times_{\mathrm{Spec}K}\mathrm{Spec} K_\omega$ with respect to $\lvert\cdot\rvert_\omega$. We define the pair $\ovl L=(L,\phi:=\{\phi_\omega\})$ as an \text{adelic line bundle} if it satisfies the conditions described in \citep[6.1]{adelic}. Here we briefly recall the definition of an adelic line bundle. For simplicity, we assume that $\forall\omega\in\Omega$, $\lvert\cdot\rvert_\omega$ is non-trivial i.e. $\lvert x\rvert_\omega\not=1$ for some $x\in K\setminus\{0\}$. We can write $\ovl L$ as $\ovl L=\ovl H_1-\ovl H_2+(\O_X, \{f_{\omega}\}_{\omega\in\Omega})$, where $\ovl H_i(i=1,2)$ are very ample line bundles equipped with Fubini-Study metric families given as in \cite[6.1.1]{adelic}, and $f_\omega$ are continuous functions. Note that for any algebraic extension $K'/K$, we naturally have an adelic structure $(K',(\Omega_{K'},\mathcal A_{K'},\nu_{K'}),\{\lvert\cdot\rvert_{\omega\in\Omega_{K'}}\})$ together with a map $\pi_{K'/K}:\Omega_{K'}\rightarrow \Omega$ such that $\pi_{K'/K}^{-1}(\omega)$ is set-theoretically identical with $(\mathrm{Spec}K')_{\omega}^{\mathrm{an}}$(for example, see \cite[3.4]{adelic} or subsection \ref{sub_height} of this paper). We say $\ovl L$ is an \textit{adelic line bundle} if \begin{enumerate}
    \item[\textnormal{(i)}] the function $(\omega\in\Omega)\mapsto \sup\limits_{x\in \xanomega}\lvert f_\omega(x)\rvert_\omega$ is $\nu$-measurable and dominated, and
    \item[\textnormal{(ii)}] for any closed point $p\in X$, let $K'$ be the residue field of $P$, the function $(\omega\in \Omega_{K'})\mapsto f_{\pi_{K'/K}(\omega)}(\omega)$ is $\mathcal A_{K'}$-measurable.
\end{enumerate}
We say $\phi$ is \textit{semipositive} if the following conditions are satisfied:
\begin{enumerate}
    \item If $\lvert\cdot\rvert_\omega$ is non-Archimedean, then $\phi_\omega$ is a uniform limit of nef model metrics.
    \item If $\lvert\cdot\rvert_\omega$ is Archimeden, then $\phi_\omega$ is of semi-positive Chern current.
\end{enumerate}

We define the adelic pushforward $\pi_*\ovl L$ of $\ovl L$ as the pair $(H^0(X,L),\xi_{\phi}:=\{\nm_{\phi_\omega}\}_{\omega\in\Omega})$, where each $\nm_{\phi_\omega}$ is the supnorm induced by $\phi_\omega$. 
Analogous to the degrees and minimal slopes of torsion free coherent sheaves, we can define the \textit{Arakelov degree} $\ardeg(\pi_*\ovl L)$ and \textit{minimal slope} $\mumin(\pi_*\ovl L)$.
We define the \textit{$\chi$-volume} of $\ovl L$ as 
\begin{equation}\label{eq_vol_chi_intro}
    \vol_\chi(\ovl L):=\limsup_{n\rightarrow+\infty}\frac{\ardeg(\pi_*(n\ovl L))}{n^{d+1}/(d+1)!}
\end{equation}
If $L$ is big and the \textit{asymptotic minimal slope} $\displaystyle\mumin^{\mathrm{asy}}(\ovl L):=\liminf\limits_{n\rightarrow+\infty}\frac{\mumin(\pi_*(n\ovl L))}{n}$ lies in $\R$, then the limit superior in \eqref{eq_vol_chi_intro} is actually a limit. This makes it possible to prove the differentiability and concavity of $\vol_\chi(\cdot)$.

Note that a continuous metric on $\O_{\xanomega}$ and a continuous function on $\xanomega$ are essentially the same thing.
Let $C_{\Omega}(X)$ be the set of continuous function families $f=\{f_\omega\}_{\omega\in\Omega}$ such that $(\O_X,f)$ is an adelic line bundle. For an adelic line bundle $\ovl L=(L,\phi)$ such that $L$ is big and $\phi$ is semipositive, we can define a linear functional $\mu_{\ovl L}:C_{\Omega}(X)\rightarrow \R$ associated to $\ovl L$. 

Let $\ovl K$ be the algebraic closure of $K$, then for any adelic line bundle $\ovl L$, we can define a height function $h_{\ovl L}:X(\ovl K)\rightarrow \R$. Conversely, for any $x\in X(\ovl K)$ we define the linear functional $\mu_x$ on $C_{\Omega}(X)$ as $\mu_x(f):=h_{(\O_X,f)}(x)$. We consider an infinite directed set $I$, that is, a set $I$ together with a binary relation $\leq$ such that 
\begin{enumerate}
    \item[\textnormal{(a)}] $\iota\leq \iota$ for any $\iota\in I$.
    \item[\textnormal{(b)}] If $\iota\leq \iota'$ and $\iota'\leq \iota''$, then $\iota\leq\iota''$.
    \item[\textnormal{(c)}] For any $\iota,\iota'\in I$, there always exists an $\iota''\in I$ such that $\iota\leq\iota''$ and $\iota'\leq\iota''$.
\end{enumerate}
Let $\{x_{\iota}\in X(\ovl K)\}_{\iota\in I}$ be a set of algebraic points indexed by $I$. We say the set is a \textit{generic net} if for any proper closed subset $Y\subsetneq X$, there exists an $\iota_0\in I$ such that for any $\iota\geq \iota_0$, we have $x_\iota\not\in Y$.

We prove an equidistribution theorem over an adelic curve after assuming the boundedness of minimal slopes of a big and semipositive adelic line bundle.
\begin{theorem}[cf. Theorem \ref{theo_equidistribution_big}]\label{theo_equi_intro}
Let $\ovl L=(L,\phi)$ be an adelic line bundle such that $L$ is big and nef, $\phi$ is semipositive and $\mumin^{\mathrm{asy}}(\ovl L)> -\infty$.
Let $\{x_\iota\in X(\ovl K)\}_{\iota\in I}$ be a generic net of algebraic points on $X$ such that 
$$\lim_{\iota\in I} h_{\ovl L}(x_\iota)= \frac{\vol_\chi(\ovl L)}{(d+1)\mathrm{vol}(L)},$$
that is, for any $\epsilon>0$, there exists $\iota_0\in I$ such that for any $\iota\geq \iota_0$, we have $\displaystyle\lvert h_{\ovl L}(x_\iota)-\frac{\vol_\chi(\ovl L)}{(d+1)\mathrm{vol}(L)}\rvert<\epsilon$.
Then for any $f\in C_\Omega(X)$, we have $\{\mu_{x_\iota}(f)\}_{\iota\in I}$ converges to $\mu_{\ovl L}(f).$
\end{theorem}
\begin{rema}
    Note that the condition that $L$ is nef is actually implied by the condition that $L$ admits semipositive metrics. Therefore we may omit this condition just for concision.
\end{rema}
In the case that $L$ is semiample, the condition $\mumin^{\mathrm{asy}}(\ovl L)> -\infty$ is automatically satisfied, a similar result was proved in \citep{huayimoriwaki_equi}. Therefore the importance of asymptotic minimal slopes rises to the surface. Note that this have been proved over number fields due to Ikoma \citep{ikoma_boundedness}, hence our result extend the equiditribution theorem in \citep{Boucksom_Growth} to the finite places of a number fields.

In this article, we will prove the boundedness of slopes for any adelic line bundle over a function field of characteristic zero. Therefore in this case, we extend the results of X. Faber\citep{Faber_2009} and W. Gubler\citep{gubler2008equidistribution}. 
\subsection*{Boundedness of minimal slopes} 
We consider a normed \textit{graded linear series}, that is, for each $m>0$, we equip a norm famliy $\xi^{(m)}=\{\nm^{(m)}_{\omega}\}$ on $H^0(X,mL)$ where $L$ is a line bundle. We further assume that each $\xi^{(m)}$ satisfies the dominancy and measurablity conditions in \citep[4.1]{adelic}, so that the minimal slope of $\ovl E_m:=(H^0(X,mL),\xi^{(m)})$ is well-defined. By imitating Ikoma's technique of estimating the successive minima over number fields, we assume the following condition:

For any $s\in E_n\setminus\{0\}$, there exist functions $(\omega\in\Omega)\mapsto \tau_\omega(s),\sigma_\omega(s)\in\R_{>0}$ depending on $s$ only such that their logarithms are $\nu$-integrable, and that we have
    \begin{equation}
        \tau_\omega(s)^{k+m}\lVert t\rVert^{(m)}_\omega\leq \lVert s^k\cdot t\rVert^{(kn+m)}_\omega \leq \sigma_\omega(s)^{k+m}\lVert t\rVert^{(m)}_\omega
    \end{equation}
for any $m\in\N$ and $t\in E_m\setminus\{0\}$.
If such a condition is satisfied, we say $\{\ovl E_m\}_{m\in\N}$ is \textit{of bounded type}. In section \ref{sect_bound}, we will prove that  $$\liminf\limits_{n\rightarrow+\infty}\frac{\mumin(\ovl E_n)}{n}>-\infty$$
for any normed graded linear series of bounded type. Moreover, if we construct our adelic curve $S$ from a projective model $B$ of a function field $K$ of characteristic $0$ as in section \ref{sec_fun_fld}, then we show that $\{\pi_*(n\ovl L)\}$ is of bounded type for any adelic curve $\ovl L$. Hence Theorem \ref{theo_equi_intro} works for any big and semipositive adelic line bundles over a function field. 

In fact, if the $\sigma$-algebra $\mathcal A$ is discrete, then for any place $\omega$ of $K$, $C(\xanomega)$ can be viewed as a subset of $C_\Omega(X)$. Indeed, we can identify $f\in C(\xanomega)$ with the family whose element is $f$ at $\omega$ and $0$ elsewhere. We set $\mu_{x,\omega}:=\rest{\mu_x}{{C(\xanomega)}}$ and $\mu_{\ovl L,\omega}:=\rest{\mu_{\ovl L}}{{C(\xanomega)}}$. They are measures on $\xanomega$ of total mass $\nu(\omega)$. 
If $K$ is a number field, all the reasoning above works well due to Remark \ref{rema_num_corps}. Therefore in the case over a function field or a number field, we see that Theorem \ref{theo_equi_intro} can be restated as follows:
\begin{theorem}[cf. Theorem \ref{theo_equi_fun} and Remark \ref{rema_num_corps}]
Let $\ovl L=(L,\phi)$ be an adelic line bundle such that $L$ is big and $\phi$ is semipositive.
Let $\{x_\iota\in X(\ovl K)\}_{\iota\in I}$ be a generic net of algebraic points on $X$ such that 
$$\lim_{\iota\in I} h_{\ovl L}(x_\iota)= \frac{\vol_\chi(\ovl L)}{(d+1)\mathrm{vol}(L)}.$$
Then for any $\omega\in\Omega$, we have $\{\mu_{x_\iota,\omega}\}_{\iota\in I}$ converges weakly to $\mu_{\ovl L,\omega}.$
\end{theorem}

Moreover, the boundedness of minimal slope will lead to the continuity of $\chi$-volume, by which we can show that $\displaystyle\frac{\vol_\chi(\ovl L)}{(d+1)\mathrm{vol}(L)}$ is nothing but the \textit{height} $h_{\ovl L}(X)$ of $X$ with respect to $\ovl L$ in the sense of intersection theory.

\subsection*{Organization of the paper}
In the first two sections, we consider the very general case, that is, equidistribution theorem over an adelic curve by assuming the boundedness of minimal slopes. In section 3, we provide an adelic version of Ikoma's method to prove the slope boundedness under certain conditions. In section 4,  we show that these conditions are satisfied over function fields. As a byproduct, we prove the continuity of $\chi$-volumes over function fields which leads to a Hilbert-Samuel formula.
\section*{Notation and conventions}
\noindent{\bf 1. } Let $K$ be a field equipped with an absolute value $\lvert\cdot\rvert$. Let $E$ be a finite-dimensional vector space over $K$ and $\nm$ be a norm on $E$.
Let $f:F\rightarrow E$ be an injective linear map of $K$-spaces, we denote by $\nm_f$ the restriction norm of $\nm$ on $F\simeq Im(f)$.
Let $G=E/Im(f)$ and $g:E\rightarrow G$ be the canonical map. Then for any $s\in G$, we define the \textit{quotient seminorm} $\lVert\cdot\rVert_{\bf q}$ by
$$\lVert s\rVert_{\bf q}:=\inf_{t\in g^{-1}(s)} \lVert t\rVert.$$
If $K$ is complete, then $\nm_{\bf q}$ is a norm.

\noindent{\bf 2. }Let $r=\dim_K E$ and $\mathrm{det}E$ be the determinant $\wedge^{r}E$ of $E$. The \textit{determinant norm} $\nm_{\mathrm{det}}$ is defined as
$$\lVert \lambda\rVert_{\mathrm{det}}=\inf_{\lambda=s_1\wedge\cdots\wedge s_r}\lVert s_1\rVert\cdots\lVert s_r\rVert.$$
Let $\lVert\cdot\rVert,\lVert\cdot\rVert'$ be norms on $E$. We define the \textit{relative volume} as
$$\mathrm{vol}(\lVert\cdot\rVert,\lVert\cdot\rVert'):=\ln\frac{\lVert\lambda\rVert'_{\mathrm {det}}}{\lVert\lambda\rVert_{\mathrm {det}}}$$
where $\lambda\in\mathrm{det}E\setminus \{0\}$.

\noindent {\bf3. }Let $k$ be a field. A $k$-variety means an integral separated scheme of finite type over $k$. Let $X$ be a projective $k$-variety, and $L$ be a line bundle over $X$. We say $L$ is \textit{strictly effective} if $h^0(X,L)>0$. Let $Y$ be a closed subscheme of $X$, we denote by $H^0(X|Y,L)$ the image of $H^0(X,L)\rightarrow H^0(Y,L|_Y)$. We say $Y$ is \textit{smooth} if the structure morphism $Y\rightarrow \mathrm{Spec}k$ is smooth. Notice that if $\dim Y=0$ then this implies that $Y$ is reduced. Let $Z$ be a closed subset of $X$. We say $Y$ \textit{avoids} $Z$ if one of the following two conditions is satisfied: \begin{enumerate}
    \item[(i)] $Y\not\subset Z$ and $\dim Y>0$. 
    \item[(ii)] $Y\cap Z=\emptyset$ and $\dim Y=0$.
\end{enumerate}
\section{Review of Arakelov geometry over adelic curves}
\subsection{Adelic curves and adelic vector bundles}\label{sub_adelic}
Let $K$ be a field and $(\Omega,\mathcal A, \nu)$ be a measure space. Let $M_K$ be the set of absolute values on $K$. If there exists a map $\varphi:(\omega\in\Omega)\mapsto \lvert\cdot\rvert_\omega\in M_K$ such that for any $a\in K\setminus\{0\}$, the function $\omega\mapsto \ln\lvert a\rvert_\omega$ is $\nu$-integrable, then we say the structure $S=(K,(\Omega,\mathcal A,\nu),\varphi)$ is an \textit{adelic curve}. Moreover, we say $S$ is \textit{proper} if the integral is always $0$. Throughout this article, we assume that $S$ is proper. For each $\omega\in\Omega$, we denote by $K_\omega$ the completion of $K$ with respect to $\lvert\cdot\rvert_\omega$.

Now let $E$ be a finite-dimensional $K$-space. For each $\omega\in\Omega$, let $\nm_\omega$ be a norm on $E\ot K_\omega$. We consider a norm family $\xi=\{\nm_\omega\}$ satisfying certain measurability and dominancy conditions \citep[4.1]{adelic}, that is, $\forall s\in E\setminus\{0\}$, $\omega\mapsto \ln\lVert s\rVert_\omega$ is $\mathcal A$-measurable and 
$$\tint_\Omega \ln\lVert s\rVert_\omega\nu(d\omega)<+\infty,$$
and the same holds for the dual $(E^\vee,\xi^\vee=\{\nm_\omega^\vee\}).$
We say the pair $\ovl E=(E,\xi)$ is an \textit{adelic vector bundle}. For any section $s\in E\setminus\{0\}$, we define the \textit{degree} of $s$ as
$$\ardeg_{\xi}(s):=-\int_\Omega\ln\lVert s\rVert_{\omega}\nu(d\omega).$$

We can consider the determinant bundle $\det \ovl E=(\det E,\det \xi:=\{\nm_{\omega,\det}\})$ where each $\nm_{\omega,\det}$ is the determinant norm. If $E\not=0$, then the \textit{Arakelov degree} of $\ovl E$ is defined as
$$\ardeg(\ovl E):=-\int_{\Omega}\ln\lVert s\rVert_{\det,\omega}\nu(d\omega)$$
where $s\in \det E\setminus\{0\}$. This definition is independent of the choice of $s$. The \textit{slope} of $\ovl E$ is defined as $\widehat\mu(\ovl E)=\ardeg(\ovl E)/\dim_K(E)$. If $E=0$, then by convention we set $\ardeg(\ovl E):=0$. 

From now on, we assume that $K$ admits a countable subfield dense in every $K_\omega$ or the $\sigma$-algebra $\mathcal A$ is discrete. In this case, for every subspace $F\subset E$, we obtain an adelic vector bundle $\ovl F$ by taking restrictions of norms. Similarly, if $E\twoheadrightarrow G$ is a surjective map, then we obtain an adelic vector bundle $\ovl G$ by taking quotient norms. The \textit{positive degree}, \textit{maximal slope} and \textit{minimal slope} are defined as 
$$\begin{cases}
\ardeg_+(\ovl E)=\sup\limits_{F\subset E}\ardeg(\ovl F),\\
\mumax(\ovl E)=\sup\limits_{0\not= F\subset E}\widehat\mu(\ovl F),\\
\mumin(\ovl E)=\inf\limits_{E\twoheadrightarrow G\not=0}\widehat\mu(\ovl G).
\end{cases}$$
Also we set by convention that $\mumax(0)=-\infty$ and $\mumin(0)=+\infty$.
\begin{prop}\label{prop_deg_posdeg}
Let $\ovl E=(E,\{\nm_{\omega}\})$ be an adelic vector bundle. We give the following properties:
\begin{enumerate}
    \item[\textnormal{(a)}] Let $f:\Omega\rightarrow\R$ be a $\nu$-integrable function. Then $\ovl {E(f)}:=(E,\{\nm_\omega \exp{(-f(\omega)})\})$ is also an adelic vector bundle and \[\begin{cases}
    \displaystyle\ardeg(\ovl {E(f)})=\ardeg(\ovl E)+(\dim_K E)\int_{\Omega} f\nu(d\omega),\\
    \displaystyle\mumax(\ovl {E(f)})=\mumax(\ovl E)+\int_{\Omega} f\nu(d\omega),\\
    \displaystyle\mumin(\ovl {E(f)})=\mumin(\ovl E)+\int_{\Omega} f\nu(d\omega).
    \end{cases}\]
    \item[\textnormal{(b)}] If $\mumin(\ovl E)\geq 0$, then $\ardeg(\ovl E)=\ardeg_+(\ovl E)$.
\end{enumerate}
\end{prop}
\begin{proof}
(a) This is due to the definition.
(b) Assume that there exists a subspace $F\subset E$ such that $\ardeg(\ovl F)>\ardeg(\ovl E)$, then $\widehat\mu(\ovl{E/F})<0$ due to \citep[Proposition 4.3.13]{adelic}, hence a contradiction.
\end{proof}
\subsection{Adelic line bundles}
Let $\pi:X\rightarrow \mathrm{Spec}K$ be a geometrically reduced projective $K$-scheme. 
Let $L$ be a line bundle over $X$. For each $\omega\in\Omega$, we denote by $X_\omega^{\mathrm{an}}$ the Berkovich analytification\citep{Berkovich} of $X_\omega:=X\times_{\mathrm{Spec}K}\mathrm{Spec}K_\omega$ with respect to $\lvert\cdot\rvert_\omega$ and $L_\omega^{\mathrm{an}}$ the analytification of $L_\omega:=L\ot_K K_\omega$. Note that as a set, $\xanomega$ consists of pairs $x=(P,\lvert\cdot\rvert_x)$ where $P$ is a scheme point in $X_\omega$ and $\lvert\cdot\rvert_x$ is an absolute value on the residue field $\kappa(P)$ extending $\lvert\cdot\rvert_\omega$. We denote by $\widehat{\kappa}(x)$ the completion of $\kappa(P)$ with respect to $\lvert\cdot\rvert_x$.
We equip each $L_\omega^{\mathrm{an}}$ with a \textit{continuous metric} $\phi_\omega$, that is, a collection $\{\lvert\cdot\rvert_{\phi_\omega}(x)\}_{x\in \xanomega}$ where each $\lvert\cdot\rvert_{\phi_\omega}(x)$ is a norm on $L\ot_{\kappa(P)}\widehat{\kappa}(x)$, such that for any regular section $s\in \Gamma(L,U)$, the function
$$(x\in U^{\mathrm{an}}_\omega)\mapsto \lvert s\rvert_{\phi_\omega}(x)$$ is continuous. If the metric family $\phi:=\{\phi_\omega\}_{\omega\in\Omega}$ satisfies certain dominancy and measurability conditions in \citep[6.1]{adelic}, then we say the pair $\ovl L=(L,\phi)$ is an \textit{adelic line bundle} over $X$. 
Let $\psi=\{\psi_\omega\}$ be another continuous metric family on $L$.
For each $\omega$, $\psi_\omega-\phi_\omega$ is a continuous metric on $\O_{X_\omega}^{\mathrm{an}}$, which corresponds to a continuous function $$(x\in \xanomega)\mapsto -\ln \lvert 1\rvert_{\psi_\omega-\phi_\omega}(x).$$
By abuse of notations, we also denote the function by $\psi_\omega-\phi_\omega$. 
We define the \textit{distance} $d(\psi_\omega,\phi_\omega):=\sup\limits_{x\in \xanomega}\lvert \psi_\omega-\phi_\omega(x)\rvert$. If $(L,\psi)$ is also an adelic line bundle, then the function $\omega\mapsto d(\psi_\omega,\phi_\omega)$ is $\nu$-integrable by definition. We denote by $d(\psi,\phi)$ the integral.

We denote by $\xi_{\phi}=\{\nm_{\phi_\omega}\}$ the norm family consisting of supnorms $\nm_{\phi_\omega}$ induced by $\phi_\omega$. Then the pair $(H^0(X,L),\xi_\phi)$ is an adelic vector bundle, which we denote by $\pi_*(\ovl L)$. The \textit{$\chi$-volume} of $\ovl L$ is defined as
$$\vol_\chi(\ovl L):=\limsup_{n\rightarrow+\infty} \frac{\ardeg(\pi_*(n\ovl L))}{n^{d+1}/(d+1)!}$$
where $d$ is the dimension of $X$.
\begin{prop}
The \textit{asymptotic maximal slope} and \textit{asymptotic minimal slope} of $\ovl L$ are defined as $$\begin{cases}
    \displaystyle\mumax^{\mathrm{asy}}(\ovl L):=\limsup_{n\rightarrow +\infty}\frac{\mumax(\pi_*(n\ovl L))}{n},\\
    \displaystyle\mumin^{\mathrm{asy}}(\ovl L):=\liminf_{n\rightarrow +\infty}\frac{\mumin(\pi_*(n\ovl L))}{n}.
\end{cases}$$
If $L$ is big and $\mumin^{\mathrm{asy}}(\ovl L)\in\R$, then the sequence $\displaystyle\left\{\frac{\ardeg(\pi_*(n\ovl L))}{n^{d+1}/(d+1)!}\right\}$ converges to $\vol_\chi(\ovl L)$.
\end{prop}
\begin{proof}
This can be proved by using \citep[Theorem 6.4.6]{adelic} and Proposition \ref{prop_deg_posdeg}. 
\end{proof}
Now let $f$ be a $\nu$-integrable function. Then we have an adelic line bundle $(\O_X,f)$ if we consider the constant function $(x\in \xanomega)\mapsto f(\omega)$ for each $\omega\in\Omega$. We denote by $\ovl L(f)$ the adelic line bundle $(L,\phi+f)$, where each $(\phi+f)_\omega$ denotes the continuous metric $\{e^{-f(\omega)}\lvert\cdot\vert_{\phi_\omega}\}_{x\in\xanomega}$. Then by Proposition \ref{prop_deg_posdeg}, we can easily see that 
\[
\begin{cases}
\displaystyle \vol_\chi(\ovl L(f))=\vol_\chi(\ovl L)+(d+1)\mathrm{vol}(L)\int_\Omega f\nu(d\omega),\\
\displaystyle\mumax^{\mathrm{asy}}(\ovl L(f))=\mumax^{\mathrm{asy}}(\ovl L)+\int_\Omega f\nu(d\omega),\\
\displaystyle\mumin^{\mathrm{asy}}(\ovl L(f))=\mumin^{\mathrm{asy}}(\ovl L)+\int_\Omega f\nu(d\omega).
\end{cases}
\]
\subsection{Adelic Cartier divisors}
This subsection is a preliminary for subsection \ref{conti_vol_fun} about the continuity of $\vol_\chi(\cdot)$. If you are not interested, you may skip this part.
Now we further assume that $X$ is normal and geometrically integral. Let $D$ be a Cartier divisor on $X$. For each $\omega\in\Omega$, we denote by $D_\omega$ the pull-back of $D$ through $X_\omega\rightarrow X$. Then there is a bijection 
$$\big\{\text{Green functions }g_\omega\text{ on }D_\omega\big\}\rightarrow \big\{\text{continuous metrics }\phi_{g_\omega}\text{ on } \O_{X_\omega}(D_\omega)\big\}$$
(for the definition of Green functions, see \citep[2.5]{adelic}). Let $g=\{g_\omega\}$ be a Green function family on $D$, that is, each $g_\omega$ is a Green function on $D_\omega$. We call such a pair $\ovl D=(D,g)$ an \textit{adelic Cartier divisor} if the corresponding pair $(\O_X(D),\phi_g:=\{\phi_{g_\omega}\})$ is an adelic line bundle. We denote by $\widehat{\Div}(X)$ the set of adelic Cartier divisors. Let $\mathbb K=\Q$ or $\R$. 
The set $\widehat{\Div}_{\mathbb K}(X)$ of adelic $\mathbb{K}$-Cartier divisors are defined as $\widehat{\Div}(X)\ot_\Z \mathbb K/N_\mathbb{K}$ where the $N_\mathbb{K}$ is the subspace spanned by elements of the form $$(0,\lambda_1 f_1+\cdots+\lambda_r f_r)\ot 1-\sum_{i=1}^{r}(0,f_i)\ot \lambda_i$$
where $\lambda_i\in \mathbb{K}$, and $f_i$ are continuous function families such that $(0,f_i)\in\widehat{\Div}(X)$.
For any $\ovl D=(D,g)\in \widehat{\Div}_{\mathbb K}(X)$, we denote by $H^0_{\mathbb K}(X,D)$ the set $$\{f\in K(X)\mid \mathrm{div}(f)+D\geq_{\mathbb K} 0\}.$$ Then we can assign a norm family $\xi_g$ on $H^0_{\mathbb K}(X,D)$ such that $(H^0_{\mathbb K}(X,D),\xi_g)$ is an adelic vector bundle \citep[Theorem 6.2.18]{adelic}.
The \textit{volume} $\vol(\cdot)$ and \textit{$\chi$-volume} $\vol_\chi(\cdot)$ on $\widehat{\Div}_\R(X)$ are given by
$$\begin{cases}
\displaystyle\vol(\ovl D):=\limsup\limits_{n\rightarrow +\infty}\frac{\ardeg_+(H^0_\R(nD),\xi_{ng})}{n^{d+1}/(d+1)!},\\
\displaystyle\vol_\chi(\ovl D):=\limsup\limits_{n\rightarrow +\infty}\frac{\ardeg(H^0_\R(nD),\xi_{ng})}{n^{d+1}/(d+1)!}.
\end{cases}$$
For reader's convenience, here we list some properties:
\begin{enumerate}
\item[(A)] If $D$ is big, then $\vol(\ovl D)$ is actually a limit \citep[Theorem 6.4.6]{adelic}.
\item[(B)] $\vol(\cdot)$ is continuous on $\widehat{\Div}_\R(X)$ in the sense of \citep[Theorem 6.4.24]{adelic}.
\item[(C)] If $D$ is big and the \textit{asymptotic minimal slope} $$\mumin^{\mathrm{asy}}(\ovl D):=\liminf_{n\rightarrow +\infty}\frac{\mumin(H^0_\R(nD),\xi_{ng})}{n}$$ is finite, then $\vol_{\chi}(\ovl D)$ is a limit.
\item[(D)] If the \textit{upper asymptotic minimal slope} $$\mumin^{\sup}(\ovl D):=\limsup_{n\rightarrow +\infty}\frac{\mumin(H^0_\R(nD),\xi_{ng})}{n}$$
is positive, then $\vol(\ovl D)=\vol_\chi(\ovl D)$.
\end{enumerate}
Note that (C) and (D) can be derived from definitions and Proposition \ref{prop_deg_posdeg}.

\section{Equidistribution theorem over adelic curves}
Let $\pi:X\rightarrow \mathrm{Spec}K$ be a geometrically reduced and projective $K$-scheme of dimension $d$.
\subsection{Differentiability and concavity of $\chi$-volume} A similar discussion of this subsection can be also found in \citep[Chapter 7]{huayimoriwaki_equi}.
Let $\ovl L=(L,\phi)$ be an adelic line bundle such that $L$ is big. We assume that $\mumin^{\mathrm{asy}}(\ovl L)>-\infty$, in which case $$\vol_\chi(\ovl L)=\limnto \frac{\ardeg(\pi_*(n\ovl L))}{n^{d+1}/(d+1)!}.$$
Note that if $\mumin^{\mathrm{asy}}(L,\phi)>-\infty$ for some continuous metric family $\phi$, then for any continuous metric family $\psi$ of $L$ such that $\ovl L'=(L,\psi)$ is an adelic line bundle, we have $\mumin^{\mathrm{asy}}(\ovl L')>-\infty$ as well.
We can see that \begin{align*}
    \vol_\chi(\ovl L)-\vol_\chi(\ovl L')&=\limnto \frac{\ardeg(\pi_*(\ovl nL))-\ardeg(\pi_*(\ovl L'^{\ot n}))}{n^{d+1}/(d+1)!}\\
    &=\limnto \frac{\displaystyle\int_\Omega \mathrm{vol}(\nm_{n\phi_\omega},\nm_{n\psi_\omega})\nu(d\omega)}{n^{d+1}/(d+1)!}.
\end{align*}
Here the relative volume $\mathrm{vol}(\nm,\nm')$ is defined in the section of notation and conventions.
Given the following facts that
\begin{enumerate}
    \item[\textnormal{(i)}] $\big\lvert\mathrm{vol}(\nm_{n\phi_\omega},\nm_{n\psi_\omega}))\big\lvert\leq d(n\phi_\omega, n\psi_\omega) h^0(nL)= d(\phi_\omega, \psi_\omega)n h^0(nL)$ by the definition of distance function.
    \item[\textnormal{(ii)}] The function $(\omega\in \Omega) \mapsto d(\phi_\omega,\psi_\omega)$ is integrable due to the definition of adelic line bundles.
    \item[\textnormal{(iii)}] The limit $\displaystyle{\mathrm{vol}(L_\omega,\phi_\omega,\psi_\omega):=\limnto \frac{\mathrm{vol}(\nm_{n\phi_\omega},\nm_{n\psi_\omega})}{n^{d+1}/d!}}$ exists due to \citep[Theorem 4.5]{Chen_Distribution}.
\end{enumerate}
The Lebesgue dominated convergence theorem shows that
$$\vol_\chi(\ovl L)-\vol_\chi(\ovl L')=(d+1)\int_\Omega \mathrm{vol}(L_\omega,\phi_\omega,\psi_\omega)\nu(d\omega).$$
In particular, $\displaystyle\lvert\vol_\chi(\ovl L)-\vol_\chi(\ovl L')\rvert\leq(d+1)\int_\Omega d(\phi_\omega,\psi_\omega)\nu(d\omega),$ which gives the following:
\begin{prop}\label{prop_vol_chi_func}
    Let $C_{\Omega}(X,L)$ be the set of continuous metric families $\phi$ such that $(L,\phi)$ is an adelic line bundle. We define the pseudometric function as
    $\displaystyle d(\phi,\psi):=\int_\Omega d(\phi_\omega,\psi_\omega)\nu(d\omega)$ for $\phi,\psi\in C_{\Omega}(X,L)$. Then $\vol_\chi(\cdot)$ is a continuous function on $C_{\Omega}(X,L)$.
\end{prop}

Now we assume that $\phi$ is semipositive, that is, each $\phi_\omega$ is semipositve \citep[2.3]{adelic}. For each $\omega\in\Omega$, we denote by $(dd^c \phi_\omega)^d$ the measure on $\xanomega$ induced by $\phi_\omega$\citep{Chambert2012forms}. Then we have the following proposition:
\begin{prop}[Differentiability of $\vol_\chi(\cdot)$]\label{diff_vol_chi} Let $\ovl L=(L,\phi)$ be an adelic line bundle such that 
\begin{enumerate}
    \item[\textnormal{(i)}] $L$ is big and $\phi$ is semipositive.
    \item[\textnormal{(ii)}] $\mumin(\ovl L)>-\infty$.
\end{enumerate}
Let $f=\{f_\omega\}_{\omega\in \Omega}$ be a family of continuous functions such that $(\O_X,f)$ is an adelic line bundle. Then
\begin{equation}\label{derivative_vol_S}
    \frac{d}{dt}\bigg|_{t=0}\vol_\chi(L,\phi+tf)=(d+1)\int_{\Omega}\left(\int_{\xanomega} f_\omega(dd^c\phi_\omega)^d\right)\nu(d\omega).
\end{equation}
\end{prop}
\begin{proof}
For each $\omega\in\Omega$, we have
\begin{equation}\label{deri_vol_local}\frac{d}{dt}\bigg|_{t=0} \mathrm{vol}(L_\omega,\phi_\omega+tf_\omega,\phi_\omega)=\int_{\xanomega} f_\omega (dd^c \phi_\omega)^d
\end{equation}
due to \citep[Theorem 1.2]{Boucksom_2021} and \citep[Theorem B]{Boucksom_Growth}.
Notice that $$\frac{1}{\lvert t\rvert}\lvert\mathrm{vol}(L_\omega,\phi_\omega+t f_\omega,\phi_\omega)\rvert\leq \sup_{x\in \xanomega} \lvert f_\omega(x)\rvert.$$
Again by dominated convergence theorem, \eqref{derivative_vol_S} can be obtained by taking an integral of \eqref{deri_vol_local} over $\Omega$.
\end{proof}

In the following, we prove the concavity of $\vol_\chi(\cdot)$.
Take a valuation $v:K(X)\rightarrow \Z^d\cup\{+\infty\}$ of rank $d$, that is, $v(xy)=v(x)+v(y)$ and $v(x+y)\geq \min(v(x),v(y))$ for $x,y\in K(X)$, where $\Z^d$ is equipped with the lexicographic order. Such a valuation can be induced by a flag of subvarieties as in \cite{Lazarsfeld2009} or by a regular rational point as in \cite[6.4.3]{adelic}. By \citep[6.3]{adelic}, we can construct a concave function
$$G_{\ovl L}:\Delta(L)\rightarrow \R$$
where $\Delta(L)\subset\R^{d}$ is the Okounkov body associated to the linear series of $L$. We have the following properties:
\begin{enumerate}
    \item[(a)] $G_{n\ovl L}(nx)=nG_{\ovl L}(x)$ for any $x\in\Delta(L)$.
    \item[(b)] If $\mumin^{\mathrm{asy}}(\ovl L)>-\infty$, then 
    $\displaystyle\vol_\chi(\ovl L)=(d+1)\int_{\Delta(L)}G_{\ovl L}(x)dx.$
    \item[(c)] Let $\ovl L_1$ and $\ovl L_2$ be adelic line bundles whose underlying line bundles $L_1$ and $L_2$ are big. Then for any $x\in\Delta(L_1)$ and $y\in \Delta(L_2)$, it holds that
    $$G_{\ovl L_1}(x)+G_{\ovl L_2}(y)\leq G_{\ovl L_1+\ovl L_2}(x+y).$$
\end{enumerate}
Then we give the following:
\begin{prop}[Concavity of $\vol_\chi(\cdot)$]\label{prop_concave_vol_chi}
Let $\ovl L=(L,\phi)$ be an adelic line bundle such that $L$ is big and $\mumin(\ovl L)>-\infty$. Let $f=\{f_\omega\}_{\omega\in \Omega}$ be a family of continuous functions such that $(\O_X,f)$ is an adelic line bundle. Then
$$(t\in\R)\mapsto\frac{\vol_\chi(L,\phi+t f)}{(d+1)\mathrm{vol}(L)}$$
is a concave function.
\end{prop}
\begin{proof}
For any $t_1,t_2\in \R$, let $\psi_1=\phi+t_1 f$ and $\psi_2=\phi+t_2f$. Let $\ovl L_1=(L,\psi_1)$ and $\ovl L_2=(L,\psi_2)$.
Now let $\alpha=p/q\in(0,1)$ be a rational number, where $p,q>0$. Notice that $$\Delta((q-p)L)\times \Delta(pL)\rightarrow \Delta(qL), (x,y)\mapsto x+y$$ is surjective, we have 
\begin{align*}\allowdisplaybreaks\displaystyle
    q\frac{\displaystyle\int_{\Delta(L)}G_{(L,(1-\alpha)\psi_1+\alpha\psi_2)}(x)dx}{\mathrm{vol}(L)}&=
    \frac{\displaystyle\int_{\Delta(qL)}G_{(q-p)\ovl L_1+p\ovl L_2}(x)dx}{\mathrm{vol}(qL)}\\
    &\geq
    \frac{\displaystyle\int_{\Delta((q-p)L)}G_{(q-p)\ovl L_1}(x)dx}{\mathrm{vol}((q-p)L)}+\frac{\displaystyle\int_{\Delta(pL)}G_{p\ovl L_2}(x)dx}{\mathrm{vol}(pL)}\\
    &=(q-p)\frac{\displaystyle\int_{\Delta(L)}G_{\ovl L_1}(x)dx}{\mathrm{vol}(L)}+p\frac{\displaystyle\int_{\Delta(L)}G_{\ovl L_2}(x)dx}{\mathrm{vol}(L)}
\end{align*}
where the equalities are due to property (a), and the inequality is due to (c). Hence 
\begin{equation}\label{eq_concavity}
\frac{\vol_\chi(L,(1-\alpha)\psi_1+\alpha\psi_2)}{(d+1)\mathrm{vol}(L)}\geq (1-\alpha)\frac{\vol_\chi(L,\psi_1)}{(d+1)\mathrm{vol}(L)}+\alpha\frac{\vol_\chi(L,\psi_2)}{(d+1)\mathrm{vol}(L)}
\end{equation}
holds for any rational number $\alpha\in (0,1)\cap \Q$.
Moreover, the continuity of $\displaystyle{\frac{\vol_\chi(L,\phi+tf)}{(d+1)\mathrm{vol}(L)}}$ due to Proposition \ref{prop_vol_chi_func} guarantees that \eqref{eq_concavity} holds for any real number $\alpha\in(0,1).$
\end{proof}
\begin{rema}
If the $\sigma$-algebra $\mathcal A$ is discrete, then for any $\omega\in\Omega$, we consider the case that $f=f\cdot \mathds{1}_{\{\omega\}}$, that is $f_{\omega'}=0$ for any $\omega'\not=\omega$, then the concavity of the function $t\mapsto \displaystyle\frac{\vol_\chi(L,\phi+tf)}{(d+1)\mathrm{vol}(L)}$ can be proved by using the local Hodge index theorem given in \citep{Boucksom_hodge} and \citep{Yuan_hodge}.
\end{rema}
\subsection{Heights and measures}\label{sub_height}
Let $K'$ be an algebraic extension of $K$, we have a canonical construction $S_{K'}=(K',(\Omega_{K'},\mathcal A_{K'},\nu_{K'}),\varphi_{K'})$ of adelic curve on $K'$ extending $S$. Here we provide a short reminder in the case that $K'/K$ is finite and separable which is enough for the setting of this article. Let $$\Omega_{K'}:=\{\lvert\cdot\rvert_{\omega'}\text{ is an absolute value on }K’\text{, extending } \lvert\cdot\rvert_\omega\text{ for some }\omega\in\Omega\}.$$ By abuse of notation, we may do not distinguish $\omega'$ and $\lvert\cdot\rvert_{\omega'}$. Then we have a canonical map $\pi_{K'/K}:\Omega_{K'}\rightarrow \Omega$ such that $\lvert\cdot\rvert_{\omega'}$ extends $\lvert\cdot\rvert_{\pi_{K'/K}(\omega')}$. The $\sigma$-algebra $\mathcal A_{K'}$ is defined to be the smallest $\sigma$-algebra such $\pi_{K'/K}$ and functions of form $(\omega'\in\Omega_{K'})\mapsto\lvert \alpha\rvert_{\omega'}$ are measurable, where $\alpha$ runs over $K'$. The measure $\nu_{K'}$ is given as 
$$\nu_{K'}(A):=\int_{\Omega}\big(\sum_{\substack{\omega'\in A,\\\pi_{K'/K}(\omega')=\omega}}\frac{[K'_{\omega'}:K_\omega]}{[K':K]}\big)\nu_{K}(d\omega)$$
where $A\in \mathcal A_{K'}$ and $K'_{\omega'}$ is the completion of $K'$ with respect to $\lvert\cdot\rvert_{\omega'}$. This measure is well-defined due to \cite[Theorem 3.3.4]{adelic}. We refer the reader to \cite[3.4]{adelic} for general cases.

For any $x\in X(\ovl K)$ defined over $K'$, that is, a morphism $x:\mathrm{Spec}K'\rightarrow X$, by abuse of notation, we may also use $x$ to denote the image of the morphism. Let $\ovl L=(L,\phi)$ be an adelic line bundle over $X$. Then the pull-back $(x^*L,x^*\phi)$ induces an adelic line bundle over $S_{K'}$. We define the \textit{height function} $h_{\ovl L}:X(\ovl K)\rightarrow \R$ as 
$$h_{\ovl L}(x):=\ardeg(x^*L,x^*\phi).$$
We can give a more explicit definition based on the construction of Berkovich spaces. For any $\omega'\in\Omega_{K'}$ with $\pi_{K'/K}(\omega')=\omega$, consider the Berkovich space $x_\omega^{\mathrm{an}}\subset \xanomega$ which is $$\{\lvert\cdot\rvert\text{ is an absolute value on the residue field of }x\text{ which extends }\lvert\cdot\rvert_\omega\}$$ as a set. Then the restriction of $\lvert\cdot\rvert_{\omega'}$ on the residue field of $x$ gives a point $p_{\omega'}\in x_\omega^{\mathrm{an}}.$

Let $s$ be a rational section of $L$ not vanishing at $x$. Then the height function can be also given by
$$h_{\ovl L}(x):=-\int_{\omega'\in\Omega_{K'}}\ln\lvert s\rvert_{\phi_{\pi_{K'/K}(\omega')}}(p_{\omega'})\nu_{K'}(d\omega').$$ We can see the following properties:
\begin{enumerate}
    \item[(a)] The definition of height function is independent of the choice of $K'$ and the section $s$.
    \item[(b)] Let $\ovl L'$ be another adelic line bundle. Then we have $h_{\ovl L+\ovl L'}(x)=h_{\ovl L}(x)+h_{\ovl L'}(x)$.
\end{enumerate}
As in \cite[6.2.2]{adelic}, we define the \textit{essential minimum} as
$$\zeta_{\mathrm{ess}}(\ovl L):=\sup_{\substack{Y\subsetneq X\\\text{closed}}}\inf_{x\in X\setminus Y(\ovl K)}h_{\ovl L}(x).$$
Now we assume that $\phi$ is semipositve.
Let $C_\Omega(X)$ be the set of dominated and measurable metric families $f=\{f_\omega\}$ on $\O_X$. We define linear functionals $\mu_{x}$ and $\mu_{\ovl L}$ on $C_\Omega(X)$ as $$\begin{cases}\mu_{x}(f)=h_{(\O_X,f)}(x),\\
\mu_{\ovl L}(f)=\displaystyle{\frac{\displaystyle{\int_{\Omega}\int_{\xanomega}}f_\omega(dd^c\phi_\omega)^d\nu(d\omega)}{\mathrm{vol}(L)}.}
\end{cases}$$

\begin{defi}
    Let $I$ be an infinite directed set. We say $\{x_\iota\in X(\ovl K)\}_{\iota\in I}$ is a \textit{generic net} of algebraic points on $X$, if for any proper closed subset $Y\subset X$, there exits $\iota_0$ such that for any $\iota>\iota_0$, $x_\iota\not\in Y$. 
\end{defi}
\begin{lemm}
    If $\{x_\iota\in X(\ovl K)\}_{\iota\in I}$ is a generic net of algebraic points on $X$, then 
    $$\liminf_{\iota\in I}h_{\ovl L}(x_\iota)\geq \zeta_{\mathrm{ess}}(\ovl L)\geq \mumax^{\mathrm{asy}}(\ovl L)\geq \frac{\vol_\chi(L,\phi)}{(d+1)\mathrm{vol}(L)}$$
\end{lemm}
\begin{proof}
    The second inequality is due to \cite[Proposition 6.4.4]{adelic}, and the third is due to the definitions. Therefore it suffices to show the first inequality. For any proper closed subset $Y\subset X$, we have $$\liminf_{\iota\in I}h_{\ovl L}(x_\iota)\geq\inf_{x\in X\setminus Y(\ovl K)}h_{\ovl L}(x)$$ since $\{x_\iota\}$ is generic. Therefore $\liminf\limits_{\iota\in I}h_{\ovl L}(x_\iota)\geq \zeta_{\mathrm{ess}}(\ovl L)$ since $Y$ is an arbitrary proper closed subset, which concludes the proof.
\end{proof}
Then we have the following equidistribution theorem.
\begin{theo}\label{theo_equidistribution_big}
Let $\ovl L=(L,\phi)$ be an adelic line bundle such that $L$ is big, $\phi$ is semipositive and $\mumin^{\mathrm{asy}}(\ovl L)> -\infty$.
Let $\{x_\iota\in X(\ovl K)\}_{\iota\in I}$ be a generic net of algebraic points on $X$ such that 
$$\lim_{\iota\in I} h_{\ovl L}(x_\iota)= \frac{\vol_\chi(L,\phi)}{(d+1)\mathrm{vol}(L)}.$$
Then for any $f\in C_\Omega(X)$, we have $\{\mu_{x_\iota}(f)\}_{\iota\in I}$ converges to $\mu_{\ovl L}(f).$
\end{theo}
\begin{proof}
Let $f$ be a dominated and measurable family on $\O_X$. We 
set $h_\iota(t)=h_{(L,\phi+tf)}(x_\iota)$ and $g(t)=\displaystyle{\frac{\vol_\chi(L,\phi+tf)}{(d+1)\mathrm{vol}(L)}}.$ Then $\displaystyle\frac{d}{dt}\bigg|_{t=0}h_\iota(t)=\mu_{x_\iota}(f)$ and $\displaystyle\frac{d}{dt}\bigg|_{t=0}g(t)=\mu_{\ovl L}(f).$
We thus conclude the proof by using \citep[Lemma 6.6]{Boucksom_Growth}.
\end{proof}
\begin{rema}
According to \citep[Proposition 2.8]{BPRM2015Dis},
the existence of such a generic net of algebraic points satisfying the condition in Theorem \ref{theo_equidistribution_big} is equivalent to say that
$$\zeta_{\mathrm{ess}}(\ovl L)=\frac{\vol_\chi(L,\phi)}{(d+1)\mathrm{vol}(L)}.$$
\end{rema}
\section{Boundedness of minimal slopes}\label{sect_bound}
In this section, we assume that $K$ is of characteristic $0$. We give an adelic version of Ikoma's proof\citep{ikoma_boundedness} for slope boundedness. 
\subsection{Normed graded linear series of bounded type}
Let $X$ be a projective and normal $K$-variety of dimension $d$. Let $L$ be a line bundle on $X$. For $n\in\N$, let $E_n:=H^0(X,nL)$. We equip each $E_n$ with a dominated and measurable norm family $\xi^{(n)}=\{\lVert\cdot\rVert^{(n)}_\omega\}_{\omega\in\Omega}$.
We say $\{\ovl E_n:=(E_n,\xi^{(n)})\}$ is of \textit{bounded type} if for any $s\in E_n\setminus\{0\}$, there exists functions  $(\omega\in\Omega)\mapsto\tau_\omega(s),\sigma_\omega(s)\in\R_{>0}$ whose logarithms are integrable, such that for any $m\in\N$ and $t\in E_m\setminus\{0\}$, we have
    \begin{equation}\label{eq_bounded_type}
        \tau_\omega(s)^{k+m}\lVert t\rVert^{(m)}_\omega\leq \lVert s^k\cdot t\rVert^{(kn+m)}_\omega \leq \sigma_\omega(s)^{k+m}\lVert t\rVert^{(m)}_\omega.
    \end{equation}

We set $\tau(s)=-\displaystyle\int_\Omega\ln \tau_\omega(s)\nu(d\omega)$ and $\sigma(s)=-\displaystyle\int_\Omega \ln \sigma_\omega(s)\nu(d\omega).$
We claim the following theorem:
\begin{theo}\label{theo_bounded_slope}
If $\{\ovl E_n\}$ is of bounded type, then $$\liminf_{n\rightarrow+\infty}\frac{\mumin(\ovl E_n)}{n}>-\infty.$$
\end{theo}
In the proof, we may need to assume that $X$ is smooth with an auxiliary line bundle $A$ which can be guaranteed by the following:
\begin{lemm}
It suffices to prove Theorem \ref{theo_bounded_slope} in the case that $X$ is smooth, and that there exists a line bundle $A$ such that
\begin{enumerate}
    \item[\textnormal{(i)}] $A$ is big and globally generated line bundle over $X$.
    \item[\textnormal{(ii)}] If $L$ is big, then there exists an integer $a$ such that $aL-A\geq 0$ is not big.
\end{enumerate}
\end{lemm}
\begin{proof}
If $L$ is not big, we just take a desingularization $f:X'\rightarrow X$ and a very ample line bundle $A$ on $X'$.
If $L$ is big. Let $a$ be the minimal positive integer such that the base locus $\mathrm{Bs}(aL)$ is stable, that is, $\mathrm{Bs}(aL)=\bigcap_{n\geq 0}\mathrm{Bs}(nL)$. Let $f:X'\rightarrow X$ be a desingularization of the blow-up of $X$ at $\mathrm{Bs}(aL)$. Consider the decomposition $$a f^*L=A+F$$ where $A$ is the moving part and $F$ is the fixed part. Here we give a brief reminder on moving and fixed parts. Let $E$ be the greatest effective divisor such for any $t\in H^0(X', af^*L)$, $\mathrm{div}(t)\geq E$. In this case, the support of $E$ is contained in the exceptional divisor of $f$. We define the fixed part as $F:=\O_{X'}(E)$ and moving part as $A=af^*L-F$. Then $F$ is effective but not big due to \citep[Claim 1.4.7]{ikoma_boundedness}. By the normality of $X$ and $X'$, in both cases, we have $$H^0(X',nf^*L)\simeq H^0(X,nL\ot f_*\O_{X'})=H^0(X,nL).$$
We conclude the proof by replacing $X$ and $L$ with $X'$ and $f^*L$ respectively. 
\end{proof}

\subsection{Proof of Theorem \ref{theo_bounded_slope}}
We fix an integer $b$ such that $bA$ contains a very ample line bundle, that is, $bA-H\geq 0$ for some very ample line bundle $H$. 

If $L$ is big, then we fix non-zero sections $t_1\in H^0(X,A)$ and $t_2\in H^0(X,aL-A)$. If $L$ is not big, we may assume that $H^0(X,m'L)\not=0$ for some $m'> 0$ (otherwise Theorem \ref{theo_bounded_slope} holds trivially), we take $0\not=t_1\in H^0(X,m'L)$ and $t_2=1\in H^0(X,\O_X)$.
\begin{prop}\label{prop_flag}
For any $0\leq i\leq d$, we denote by $A_i:=2^i bA$.
There exists a flag $$Y_0\subset Y_1\subset\cdots\subset Y_d=X$$ of closed subvarieties defined inductively as $Y_i:=\mathrm{div}(s_i|_{Y_{i+1}})$ where $0\leq i<d$ and $s_i\in H^0(X,A_i)$, such that
\begin{enumerate}
    \item[\textnormal{(Y1)}] $Y_i$ is smooth.
    \item[\textnormal{(Y2)}] $Y_i$ avoids $\lvert\mathrm{div}(t_1)\rvert\cup\lvert\mathrm{div}(t_2)\rvert$.
    \item[\textnormal{(Y3)}] If $i\geq 1$, $H^0(X,kbA)\rightarrow H^0(Y_i,kbA|_{Y_i})$ is surjective for $0<k<2^{i}$.
\end{enumerate}
Here we use terminologies in Notation and conventions {\bf3}.
\end{prop}
\begin{proof}
We construct the flag $Y_0\subset\cdots\subset Y_d$ by a backward induction on $i$. Assume that for a fixed $0\leq i\leq d-1$, we can take a sequence of sections $\{s_j\in H^0(X,A_j)\}_{i+1\leq j<d}$ such that $Y_{d}=X,Y_{d-1}=\mathrm{div}(s_{d-1})\cdots, Y_{i+1}=\mathrm{div}(s_{i+1}|_{Y_{i+2}})$ satisfies (Y1)(Y2)(Y3).
In particular, $$H^0(X,A_i)\rightarrow H^0(Y_{i+1},A_i|_{Y_{i+1}})$$ is surjective due to (Y3). Hence it suffices to find a non-zero section $s_i\in H^0(Y_{i+1},A_i|_{Y_{i+1}})$ such that $Y_{i}:=\mathrm{div}(s_i)$ satisfying (Y1)-(Y3). Since $A_i|_{Y_{i+1}}$ is free and contains a very ample line bundle for $0\leq i<d$, by the Bertini's theorem, there exists a section $s_i$ such that $Y_i$ satisfies (Y1) and (Y2). We will see that if $i\geq 1$, then (Y3) is then automatically satisfied. Consider the exact sequence
$$0\rightarrow (k-2^i)bA|_{Y_{i+1}}\rightarrow kbA|_{Y_{i+1}}\rightarrow kbA|_{Y_i}\rightarrow 0$$
for $0<k<2^i$. It suffices to show that $H^1(Y_{i+1},(k-2^i)bA|_{Y_{i+1}})=0$. Let $K_{Y_{i+1}}$ be the dualizing sheaf over $Y_{i+1}$. By Serre's duality, we have $H^1(Y_{i+1},(k-2^i)bA|_{Y_{i+1}})=H^i(Y_{i+1},K_{Y_{i+1}}+(2^i-k)bA|_{Y_{i+1}})$ which vanishes due to Kawamata-Viehweg vanishing theorem.
\end{proof}

We take an integer $c$ such that $(cA-L)|_{Y_i}$ is strictly effective for $0\leq i\leq d$, i.e. $h^0(Y_i,(cA-L)\mid_{Y_i})>0$.
\begin{prop}\label{prop_bounded_slope}
Let $s\in H^0(X,A_{i})$ be a non-zero section. 
For each $m,k\geq 0$, we denote by $\nm^{(m)}_{\omega,\cdot s^k,\bf q}$ the norm on $H^0(X|Y_i,mL-kA_{i})$ given by $\nm^{(m)}_\omega$ and
$$\begin{cases}
0\rightarrow H^0(X,mL-kA_{i})\xrightarrow{\cdot s^k} H^0(X,mL)\\
 H^0(X,mL-kA_{i})\rightarrow H^0(X|Y_i,mL-kA_{i})\rightarrow 0.
\end{cases}$$
Denote that $\xi^{(m)}_{\cdot s^k,\bf q}=\{\nm^{(m)}_{\omega,\cdot s^k,\bf q}\}$.
Then 
$$\mumin(H^0(X|Y_i,mL-kA_{i}),\xi^{(m)}_{\cdot s^k,\bf q})\geq C(m+k)+T$$
for some constants $C,T$ depending on $s, t_1, t_2, Y_0,\dots,Y_d, a,b,c, A$ and $\{\ovl E_n:=(E_n,\xi^{(n)})\}$.
\end{prop}
\begin{proof}
{\bf Step 1.} We show that if there is another $s'\in H^0(X,A_{i})\setminus\{0\}$, then there exists a constant $D(s,s')$ such that \begin{align}\label{ineq_section_not_depend}
    &\mumin(H^0(X|Y_i,mL-kA_{i}),\xi^{(m)}_{\cdot s^k,\bf q})\geq\\
    &\kern5em\mumin(H^0(X|Y_i,mL-kA_{i}),\xi^{(m)}_{\cdot s'^k,\bf q})+D(s,s')(m+k)
\end{align}
Note that if $L$ is not big, then \eqref{ineq_section_not_depend} holds trivially since $H^0(X|Y_i,mL-kA_i)=0$ if $k>0$.
If $L$ is big, we denote by $\delta_\omega$ the operator norm of $(H^0(X,mL-kA_i),\lVert\cdot\rVert^{(m)}_{\cdot s'^k,{\bf q}})\rightarrow(H^0(X,mL-kA_i),\lVert\cdot\rVert^{(m)}_{\cdot s^k,{\bf q}}).$
Take $r=t_2^{2^{i}b}\in H^0(X,2^{i}baL-A_{i})$. 
Then the following commutative diagram
\[ \begin{tikzcd}
H^0(X,mL) \arrow{r}{\cdot (s' r)^k}  & H^0(X,(m+2^{i-1}bak)L)  \\%
H^0(X,mL-kA_{i} \arrow[swap]{u}{\cdot s^k}) \arrow{r}{\cdot s'^k}& H^0(X,mL). \arrow{u}{\cdot (sr)^k}
\end{tikzcd}
\]
and \eqref{eq_bounded_type} yield that 
$$\delta_\omega\leq (\sigma_\omega (sr) \tau_\omega(s' r)^{-1})^{m+k}.$$
By \cite[Proposition 4.3.31(2)]{adelic}, we can set that $D(s,s')=\sigma(sr)-\tau(s'r).$

{\bf Step 2.} We proceed on induction of $i$. To be more precise, we will show that if the proposition holds for any $i<d$, then it holds for $i+1$. 

We set $c'=\lceil c/(2^{i}b)\rceil+1$.
For any $\rho\in\N$ and $0\leq i< d$, we denote by $\rho Y_i(\text{resp. }\rho \mathrm{div}(s_i))$ the closed subscheme $\mathrm{div}((s_i|_{Y_{i+1}})^{\rho})(\text{resp. }\mathrm{div}(s_i^{\rho}))$.
At first we can show that for any $0\leq k<c'm$,
\begin{align}
    H^0(X|Y_{i+1},mL-kA_i)&\simeq H^0(X|(c'm-k)Y_i,mL-kA_i)\label{eq_dim_red}\\
    &\simeq \mathop{\bigoplus}\limits_{k\leq l\leq c'm}H^0(X|Y_i,mL-lA_i)\label{eq_decomp}.
\end{align}
Indeed, for any $k<c'm$, consider the exact sequence
$$0\rightarrow m(L-c'A_{i})\xrightarrow{\cdot s_i^{c'm-k}} mL-kA_i\rightarrow (mL-kA_{i})|_{(c'm-k)\mathrm{div}(s_i)}\rightarrow 0.$$
We can see that the kernel
\begin{align*}
    &\mathrm{Ker}(H^0(X|Y_{i+1},mL-kA_i)\rightarrow H^0(X|(c'm-k)Y_{i},mL-kA_i))
\end{align*}
is a subspace of $H^0(X|Y_{i+1},m(L-c'A_i))$.
Notice that $((c'-1)A_i-L)|_{Y_{i+1}}$ is strictly effective, we have $H^0(X|Y_{i+1},m(L-c'A_i))=0$, which implies \eqref{eq_dim_red}.
Hence we only need to consider $(H^0(X|nY_{i},mL-kA_i),\xi^{m}_{\cdot s_i^k,\bf q})$ for $0<n\leq c'm$.
By \citep[Claim 3.5.8]{moriwaki2009convol}, there exists an injective homomorphism  
$$j_n:mL-(k+n)A_i|_{Y_i}\rightarrow mL-kA_i|_{(n+1)Y_i}$$
such that the following diagram is commutative:
\[ \begin{tikzcd}[every label/.append
  style={font=\tiny}]
0\arrow{r} &mL-(k+n)A_i \arrow{r}{\cdot s_i^n} \arrow{d} & mL-kA_i\arrow{r}\arrow{d} &(mL-kA_i)|_{n\mathrm{div}(s_i)}\arrow{r}\arrow{d} &0 \\
0\arrow{r} &(mL-(k+n)A_i)|_{Y_{i+1}} \arrow{r}{\cdot (s_i|_{Y_{i+1}})^n} \arrow{d} & (mL-kA_i)|_{Y_{i+1}} \arrow{r}\arrow{d} &(mL-kA_i)|_{nY_i}\arrow{r}\ar[equal]{d} &0 \\%
0\arrow{r} &(mL-(k+n)A_i)|_{Y_i} \arrow{r}{j_n}  & (mL-kA_i)|_{(n+1)Y_i} \arrow{r} & (mL-kA_i)|_{nY_i}\arrow{r} &0
\end{tikzcd}
\]
whose horizontal rows are exact.
Then we have the exact sequence
\begin{align*}
    &0\rightarrow H^0(X|Y_{i},mL-(k+n)A_i)\xrightarrow{j_n} H^0(X|(n+1)Y_{i},mL-kA_i)\\
    &\kern20em\rightarrow H^0(X|nY_{i},mL-kA_i)\rightarrow 0
\end{align*}
which implies \eqref{eq_decomp}.
Moreover, consider the commutative diagram:
\[ \begin{tikzcd}
H^0(X,mL-(k+n)A_i) \arrow{d}{\mathrm{quot}} \arrow{r}{\cdot s_i^n}& H^0(X,mL-kA_i) \arrow{d}{\mathrm{quot}}\\
H^0(X|Y_{i},mL-(k+n)A_i) \arrow{r}{j_n}  & H^0(X|(n+1)Y_{i},mL-kA_i).\end{tikzcd}
\]
It always holds that
$\nm^{(m)}_{\omega,\cdot s_i^k,{\bf q},j_n}\leq \nm^{(m)}_{\omega,\cdot s_i^{k+n},\bf q}$ for each $\omega\in\Omega$. We denote by $\xi^{(m)}_{\cdot s_i^k,{\bf q},j_n}$ the norm family $\left\{\nm^{(m)}_{\omega,\cdot s_i^k,{\bf q},j_n}\right\}_{\omega\in\Omega}$ on $H^0(X|{Y_i},mL-(k+n)A_i)$.
Then \begin{align*}
    &\mumin(H^0(X|{Y_i},mL-(k+n)A_i),\xi^{(m)}_{\cdot s_i^k,{\bf q},j_n})\\&\kern10em\geq \mumin(H^0(X|Y_{i},mL-(k+n)A_i),\xi^{(m)}_{\cdot s_i^{k+n},\bf q}).
\end{align*}
By the induction hypothesis, there exists constants $C_i,T_i$ such that $$\mumin(H^0(X|Y_{i},mL-kA_i),\xi^{(m)}_{\cdot s_i^{k},\bf q})\geq C_i (m+k)+T_i.$$
Hence \eqref{eq_dim_red} and \eqref{eq_decomp} give that \begin{align*}
    \mumin&(H^0(X|Y_{i+1},mL-kA_i),\xi^{(m)}_{\cdot s^k_i,\bf q})\\
    &\geq \min\limits_{k\leq l\leq c'm}\{\mumin(H^0(X|Y_{i},mL-lA_i),\xi^{(m)}_{\cdot s_i^{l},\bf q}) \}\\
    &\geq \min(C_i,0)(c'+1)m+T_i
\end{align*}
where the first inequality is due to \citep[Proposition 4.3.33]{adelic}. 

{\bf Step 3.} Now it suffices to prove the case that $i=0$. Notice that $\dim_K H^0(X|Y_0,mL-kA)$ is bounded, our statement can be proved by using \citep[Claim 1.4.5]{ikoma_boundedness}. We include the claim with an explicit description:
\begin{claim}
    There exists $m_0, k_0>0$ such that for any $m,k\geq 0$, there are $0\leq m'\leq m_0$ and $0\leq k'\leq k_0$ such that 
\begin{comment}we set
\begin{align*}
    k'=\begin{cases}
    0&\text{, if }\lfloor m/a\rfloor-k\geq \lfloor m_0/a\rfloor-k_0,\\
    k_0&\text{, if }\lfloor m/a\rfloor-k<\lfloor m_0/a\rfloor-k_0.
\end{cases}\\
    m'=\begin{cases}
    m-a&\text{, if }\lfloor m/a\rfloor-k\geq \lfloor m_0/a\rfloor-k_0,\\
    k_0&\text{, if }\lfloor m/a\rfloor-k<\lfloor m_0/a\rfloor-k_0.
    \end{cases}
\end{align*}
\end{comment}
$$H^0(X|Y_0, m'L-k'A)\xrightarrow{\cdot (t_1 t_2)^{p} t_2^{ q}}H^0(X|Y_0, mL-kA)$$
is an isomorphism for some $m\geq p\geq 0$ and $q=k-k'$.
\end{claim}
We are thus done since
\begin{align*}
    &\mumin(H^0(X|Y,mL-kA),\xi^{(m)}_{\cdot t_1^k,\bf q})\\
    &\kern4em\geq \mumin(H^0(X|Y,m'L-k'A),\xi^{(m')}_{\cdot t_1^{k'},\bf q})+\min\{0,\sigma(t_1t_2)\}({m+k})\\
    &\kern4em\geq \min_{\substack{0\leq m'\leq m_0\\
    0\leq k'\leq k_0}}(\mumin(H^0(X|Y,m'L-k'A),\xi^{(m')}_{\cdot t_1^{k'},\bf q}))+\min\{0,\sigma(t_1t_2)\}({m+k})
\end{align*}
and $\mumin(H^0(X|Y,mL-kA),\xi^{(m)}_{\cdot s^k,\bf q})\geq \mumin(H^0(X|Y,mL-kA),\xi^{(m)}_{\cdot t_1^k,\bf q})+D(s,t_1)(m+k)$ by {\bf Step 1}.
\end{proof}

\section{Applications to function fields}\label{sec_fun_fld}
In this section, let $B$ be an $e$-dimensional, normal and projective variety over a base field $k$ of characteristic $0$. Let $K$ be the function field of $B$. We denote by $B^{(1)}$ the set of $1$-codimensional points in $Y$. Then each $\omega\in B^{(1)}$ gives an absolute value $\lvert\cdot\rvert_\omega:=\exp(-\mathrm{ord}_\omega(\cdot))$ on $K$ where $\mathrm{ord}_\omega(\cdot)$ is the order at $\omega$.
Let $\mathcal H=\{\mathscr H_1,\dots,\mathscr H_{e-1}\}$ be a collection of ample line bundles over $B$. Then we can equip $B^{(1)}$ with the discrete $\sigma$-algebra, and the measure $\nu(\cdot)$ such that $$\nu(\{\omega\}):=c_1(\mathscr H_1)\cdots c_1(\mathscr H_{e-1})[\omega]$$ for any $\omega\in B^{(1)}$. We therefore obtain a proper adelic curve $S$ from $B$. Since every place $\omega$ is non-Archimedean and non-trivial, we denote by $K_\omega^\circ$ the valuation ring of $K_\omega$, and by $K_\omega^{\circ\circ}$ the unique maximal ideal in $K_\omega^\circ$.

Throughout this section, we fix a projective and normal variety $\pi:X\rightarrow\mathrm{Spec}K$ of dimension $d$.
\subsection{Algebro-geometric setting vs adelic setting}
Let $\mathscr E$ be a coherent sheaf on $B$. 
We define the \textit{degree} of $\mathscr E$ as $$\degH(\mathscr E):=c_1(\mathscr H_1)\cdots c_1(\mathscr H_{e-1})c_1(\mathscr E)$$
where $c_1(\mathscr E)$ is defined as the cycle class of the determinant bundle $\det \mathscr E$ which is well-defined as in \cite[1.10]{moriwaki2014arakelov}.
The \textit{slope} $\mu(\mathscr E)$ of $\mathscr E$ is the quotient $\degH(\mathscr E)/\rank(E)$. The \textit{algebraic minimal slope} $\mu_{\min}(\mathscr E)$ is defined as
$$\mu_{\min}(\mathscr E)=\inf\{\mu(\mathscr G)\mid\mathscr G\text{ is a torsion free quotient sheaf of }\mathscr E\}.$$
The classical Harder-Narasimhan filtration theory can be applied to show that $\mu_{\min}(\mathscr E)$ is finite if $\mathscr E$ is a non-zero torsion free coherent sheaf. By convention, we set $\mu_{\min}(0):=+\infty$.

We can actually take an adelic point of view of this.
For any $\omega\in B^{(1)}$, we can equip a lattice norm $\nm_{\mathscr E,\omega}$ on $\mathscr E\ot K_\omega$\citep[4.6.3.1]{adelic}. The norm family $\xi_{\mathscr E}:=\{\nm_{\mathscr E,\omega}\}$ on $\mathscr E\ot K$ is automatically measurable and integrable. By the definition of first chern class, we can see $\ardeg(\mathscr E\ot_{B} K,\xi_{\mathscr E})=\degH(\mathscr E)$. Moreover we have the following:
\begin{lemm}
    Let $\mathscr E$ be a torsion free coherent sheaf over $B$. Then $\mumin(\mathscr E\ot_{B} K,\xi_{\mathscr E})=\mu_{\min}(\mathscr E)$.
\end{lemm}
\begin{proof}
    By \cite[Proposition 1.3.1]{moriwaki2008torsionfree}, there is a bijection between
    $$\{\text{saturated subsheaves of }\mathscr E,\text{ i.e. subsheaves }\mathscr F\text{ such that }\mathscr E/\mathscr F\text{ is torsion free}\}$$
    and $$\{\text{subspaces of }\mathscr E\ot_{B} K\}$$
    given by $\mathscr F\mapsto \mathscr F\ot_B K$.
    Therefore the map $\mathscr G\mapsto \mathscr G\ot_B K$ is a bijection between the set of torsion free quotient sheaves of $\mathscr E$ and the quotient spaces of $\mathscr E\ot_B K$. Since each quotient norm $\nm_{\mathscr E,\omega,{\bf q}}$ on $\mathscr G\ot_B K$ coincides with the lattice norm $\nm_{\mathscr G,\omega}$ due to \cite[Theorem 4.5]{luo2021relative}, we obtain that $$\widehat\mu(\mathscr G\ot_B K, \{\nm_{\mathscr E,\omega,{\bf q}}\})=\mu(\mathscr G),$$ which concludes the proof.
\end{proof}

Now let $L$ be a line bundle over $X$. Consider a $B$-model $(p:\scrX\rightarrow B,\shfL)$ of $(X,L)$, that is, $\scrX\times_{B}\mathrm{Spec}K\simeq X$ and $\shfL|_X\simeq L$. For each $\omega\in B^{(1)}$, we obtain a model metric $\phi_{\shfL,\omega}$ due to \citep[Proposition 7.5]{gubler1998local}. Then the pair $\ovl L:=(L,\phi_\shfL:=\{\phi_{\shfL,\omega}\})$ is an adelic line bundle.
By \citep[Lemma 6.3 and Theorem 6.4]{boucksom2021spaces}, for each $\omega\in\Omega$, there exists a constant $C_\omega\geq 1$ depending on the model $\scrX$ only such that 
$$\nm_{\phi_{\shfL,\omega}}\leq \nm_{p_*\shfL,\omega}\leq C_\omega\nm_{\phi_{\shfL,\omega}}
$$
where $\nm_{\phi_{\shfL,\omega}}$ is the supnorm induced by $\phi_{\shfL,\omega}$. Note that $C_\omega=1$ if the special fiber $\scrX_\omega:=\scrX\times_B \mathrm{Spec}(K_\omega^{\circ}/K_\omega^{\circ\circ})$ is reduced.
Since the generic fiber $X$ is reduced, $\scrX_\omega$ is reduced except for finitely many $\omega$\cite[12.2.4(v)]{egaIV3}. 
Hence there exists a constant $C_{\scrX}:=\sum_{\omega\in B^{(1)}}\nu(\{\omega\})\ln C_\omega$ depending only on $\scrX$ such that 
\begin{equation}\label{eq_comparison}
    \begin{cases}\ardeg(\pi_*\ovl L)\geq \degH(p_*\shfL)\geq \ardeg(\pi_*\ovl L)-C_{\scrX}h^0(X,L),\\
\mumin(\pi_*\ovl L)\geq \mu_{\min}(p_*\shfL)\geq \mumin(\pi_*\ovl L)-C_{\scrX}.
\end{cases}
\end{equation}
\subsection{Slope boundedness}
Let $(p:\scrX\rightarrow B,\shfL)$ be a $B$-model of $(X,L)$ such that $\scrX$ is normal. For each $n\geq 0$, let $\xi^{(n)}$ be the supnorm family induced by the model metric given by $n\shfL$.
We are going to show that $\{(H^0(X,nL),\xi^{(n)})\}$ is of bounded type.

For each $\omega\in B^{(1)}$, let $\scrX_\omega$ be the special fiber of $\scrX_{\mathrm{Spec}K_\omega^\circ}$. Let $\Gamma(\scrX_{\mathrm{Spec}K_\omega^\circ})$ be the Shilov boundary\cite[Proposition 2.4.4]{Berkovich}, that is, the reverse image of generic points of $\scrX_\omega$ through the reduction map $\xanomega\rightarrow \scrX_\omega$. Note that 
\begin{enumerate}
    \item[\textnormal{(a)}] $\Gamma(\scrX_{\mathrm{Spec}K_\omega^\circ})$ is finite\cite[Lemma 4.8]{boucksom2021spaces}.
    \item[\textnormal{(b)}] Each $x\in \Gamma(\scrX_{\mathrm{Spec}K_\omega^\circ})$ corresponds to a valuation on $\kappa(\eta)$, where $\eta$ is a generic point of $X_{K_\omega}$ and $\kappa(\eta)$ is the residue field.
\end{enumerate}
Hence for any $s\in H^0(X_{K_\omega},nL_{K_\omega})\setminus\{0\}$, we can set that $$\begin{cases}
\tau_\omega(s)=\min\{1,\min\limits_{\substack{x\in\Gamma(\scrX_{\mathrm{Spec}K_\omega^\circ})\\\lvert s\rvert_{n\shfL,\omega}(x)\not=0}}\{\lvert s\rvert_{n\shfL,\omega}(x)\}\}>0\\
\sigma_\omega(s)=\max\{1,\max\limits_{x\in\Gamma(\scrX_{\mathrm{Spec}K_\omega^\circ})}\lvert s\rvert_{n\shfL,\omega}(x)\}=\max\{1,\lVert s\rVert_{n\shfL,\omega}\}>0
\end{cases}$$
where the maximum and minimum can be obtained due to (a) above.
We say $r\in H^0(\scrX_{\mathrm{Spec}K_\omega^\circ},\shfL\ot K_\omega^\circ)$ is a \textit{relatively regular section} if $\mathrm{div}(r)$ is flat over $\mathrm{Spec}K_\omega^{\circ}$. Our definition is slightly different with it in \citep[A.6]{boucksom2021spaces}, but they are actually the same in the discretely valued case. Now for any $s\in H^0(X,nL)\setminus\{0\}$, we consider $s$ as a rational section of $\shfL$ over $\scrX$ since $K(X)\simeq k(\scrX)$. Then we have the decomposition $\mathrm{div}(s)=D_h+ D_{v}$ where $D_h$ is the horizontal part and $D_v$ is the part whose image under $p$ is of codimension at least $1$ in $B$. Therefore for any $\omega\not\in p(D_v)\cap B^{(1)}$, $s$ corresponds to a relatively regular section in $H^0(\scrX_{\mathrm{Spec}K_\omega^\circ},\shfL\ot K_\omega^\circ)$.
Note that $\lvert s\rvert_{\shfL,\omega}\equiv1$ on $\Gamma(\scrX_{\mathrm{Spec}K_\omega^\circ})$ if $s$ corresponds to a relatively regular section in $H^0(\scrX_{\mathrm{Spec}K_\omega^\circ},\shfL\ot K_\omega^\circ)$ due to \citep[Lemma 8.19]{boucksom2021spaces}. Hence $\tau_\omega(s),\sigma_\omega(s)$ are $1$ for all but finitely many $\omega\in B^{(1)}$, which implies that their logarithms are integrable. 
By the maximum modulus principle, for any $t\in H^0(X,mL)$, we have
\begin{align*}
    \lVert s^k t\rVert_{(kn+m)\shfL,\omega}&=\max\limits_{x\in\Gamma(\scrX_{\mathrm{Spec}K_\omega^\circ})}\lvert s^k t\rvert_{(kn+m)\shfL,\omega}(x)\\
    &\geq \max\limits_{x\in\Gamma(\scrX_{\mathrm{Spec}K_\omega^\circ})}\lvert t\rvert_{m\shfL,\omega}(x)\min\limits_{\substack{x\in\Gamma(\scrX_{\mathrm{Spec}K_\omega^\circ})\\\lvert s\rvert_{n\shfL,\omega}(x)\not=0}}\lvert s\rvert ^k_{n\shfL,\omega}(x)\\
    &\geq \tau_\omega(s)^{m+k}\lVert t\rVert_{m\shfL,\omega}.
\end{align*}
The other inequality is obvious. Therefore $\{(H^0(X,nL),\xi^{(n)})\}$ is of bounded type which gives the following:
\begin{theo}\label{theo_bound_slope_fun}
    Let $\ovl L=(L,\psi)$ be an adelic line bundle over $X$. Then $\mumin^{\mathrm{asy}}(\ovl L)\geq -\infty$.
\end{theo}
\begin{proof}
    Take a model $(\scrX,\shfL)$ of $(X,L)$ as above. 
    We have $$\mumin^{\mathrm{asy}}(\ovl L)\geq \mumin^{\mathrm{asy}}(L,\phi_{\shfL})-d(\psi,\phi_\shfL)>-\infty$$
where the second inequality is due to Theorem \ref{theo_bounded_slope}.
\end{proof}

In conclusion, the assumption on minimal slopes is automatically satisfied if we consider the specific case of Theorem \ref{theo_equidistribution_big} over a function field. Moreover, since we equip $B^{(1)}$ with the discrete $\sigma$-algebra, for each $\omega\in B^{(1)}$, we have the following:
$$\{f\in C_{B^{(1)}}(X)\mid f\cdot \mathds{1}_{\{\omega\}}=f\}\simeq C(\xanomega).$$
By this identification, both of $\mu_{x,\omega}:=\rest{\mu_x}{C(\xanomega)}$ and $\mu_{\ovl L,\omega}:=\rest{\mu_{\ovl L}}{C(\xanomega)}$ can be viewed as measures on $\xanomega$ with total mass $\nu(\{\omega\})$. Then Theorem \ref{theo_equidistribution_big} can be restated as follows:
\begin{theo}\label{theo_equi_fun}
Let $\ovl L=(L,\phi)$ be an adelic line bundle such that $L$ is big and nef, $\phi$ is semipositive.
Let $\{x_\iota\in X(\ovl K)\}_{\iota\in I}$ be a generic net of algebraic points on $X$ such that 
$$\lim_{\iota\in I} h_{\ovl L}(x_\iota)= \frac{\vol_\chi(\ovl L)}{(d+1)\mathrm{vol}(L)}.$$
Then for any $\omega\in B^{(1)}$, we have $\{\mu_{x_\iota,\omega}\}_{\iota\in I}$ converges weakly to $\mu_{\ovl L,\omega}.$
\end{theo}
\begin{rema}\label{rema_num_corps}
This technique of choosing $\tau_\omega$ and $\sigma_\omega$ can be also applied to the finite places of a number field. We refer the reader to \citep[3.2.2]{adelic} for the construction of adelic curves for number fields. We may replace the supnorm by a $L^2$-norm at infinite places for which we obtain $\tau_\omega$ and $\sigma_\omega$ due to \citep[Claim 1.2.4]{ikoma_boundedness}. Since the supnorm and $L^2$-norm are comparable due to the Gromov's inequality\citep[Lemma 30]{gillet1992arithmetic}. Therefore we have the boundedness of minimal slopes over number fields. This can be also obtained by just using \citep[Theorem 1.2.3]{ikoma_boundedness} and Minkowski's theorem.
Hence if we consider Theorem \ref{theo_equi_intro} over a number field, then we recover the equidistribution described in \citep{Boucksom_Growth}. Moreover, our result includes finite places as well.
\end{rema}

\subsection{Continuity of $\chi$-volume over function fields}\label{conti_vol_fun}
Let $\ovl L_1=(L_1,\phi_1),\cdots, \ovl L_r=(L_r,\phi_r)$ be adelic line bundles over $X$.
Let $a=(a_1,\cdots,a_r)\in \Z^r$. We set that 
$$\begin{cases}
\lVert a\rVert_1:=\lvert a_1\rvert+\cdots+\lvert a_r\rvert,\\
a\cdot {\bf\ovl L}=a_1 \ovl L_1+\cdots+a_r\ovl L_r,\\
a\cdot {\bf \phi}=a_1 \phi_1+\cdots+a_r\phi_r,\\
a\cdot{\bf L}:=a_1 L_1+\cdots+a_r L_r.
\end{cases}$$
\begin{prop}
There exists constants $S$ and $T$ such that 
$$\mumin(\pi_*(a\cdot {\bf\ovl L}))\geq \lVert a\rVert_1\cdot S+T.$$
\end{prop}
\begin{proof}

Let $p:\scrX\rightarrow B$ be a normal $B$-model of $X$. Let $\shfL_1,\cdots,\shfL_r$ be line bundles over $\scrX$ such that $\shfL_i|_X\simeq L_i$ for each $i$.
Let $\delta=\max\limits_{1\leq i\leq r}\{d(\phi_i,\phi_{\shfL_i})\}$.
Let $\mathcal V=\shfL_1 \oplus \cdots \oplus \shfL_r$ and $V=L_1\oplus \cdots \oplus L_r$. Then $\mathrm{bl}:\mathbb{P}_\scrX(\mathcal V):=Proj_{\scrX}(\mathrm{Sym}(\mathcal V))\rightarrow B$ is a $B$-model of $\mathbb{P}_{X}(V):=Proj_{X}(\mathrm{Sym}(V)),$ where $\mathrm{Sym}(\cdot)$ denotes the graded symmetric algebra\cite[Exercise 2.5.16]{Hartshorne}. We consider the tautological line bundle $\shfM:=\O_{\mathbb{P}_\scrX(\mathcal V)}(1)$ which is a $B$-model of $M:=\O_{\mathbb{P}_{X}(V)}(1)$. Denote that 
$$a\cdot {\shfL}:=a_1\shfL_1+\cdots+a_r\shfL_r$$
for $a\in \N^{r}.$
Since $\mathcal V$ is locally free, we have
$$\mathrm{bl}_*(m\shfM)=\mathop{\bigoplus}\limits_{\lVert a\rVert_1=m,a\in \N^r} p_*(a\cdot\shfL).$$
Then Theorem \ref{theo_bound_slope_fun} and \eqref{eq_comparison} show that there exists constants $S_0,T_0$ such that \begin{equation}\label{eq_min_fun_bound_uniform}
    \mu_{\min}(p_*(a\cdot \shfL))\geq \mu_{\min}(\mathrm{bl}_*(\lVert a\rVert_1\shfM))\geq \lVert a\rVert_1\cdot S_0+T_0
\end{equation}
for any $a\in \Z^r_{\geq 0}$. For each $\sigma:\{1,2,\cdots, r\}\rightarrow \{1,-1\}$, we may replace $\shfL_i$ by $\sigma(i)\shfL_i$ for $1\leq i\leq r$, and apply the same reasoning, we can see that \eqref{eq_min_fun_bound_uniform} still holds for some other constants $S_0(\sigma)$ and $T_0(\sigma)$.
By abuse of notation, we replace $S_0$ and $T_0$ by $\min_{\sigma} S_0(\sigma)$ and $\min_{\sigma}T_0(\sigma)$ respectively. Then \eqref{eq_min_fun_bound_uniform} holds for any $a\in \Z^r$.

We denote by $\xi_{a\cdot \shfL}$ the supnorm family on $H^0(X,a\cdot{\bf L})$ induced by the model metric family $\phi_{a\cdot\shfL}$. Then by the comparison \eqref{eq_comparison}, there exists a constant $C_\scrX$ depending only on $\scrX$ such that for any $a\in\Z^r$
$$\mumin(H^0(X,a\cdot{\bf L}),\xi_{a\cdot\shfL})\geq\mu_{\min}(p_*(a\cdot \shfL))-C_\scrX.$$
We set $$\delta:=\max\limits_{1\leq i\leq r}d(\phi_i,\phi_{\shfL_i}).$$
Then \begin{align*}
    \mumin(\pi_*(a\cdot{\bf\ovl L}))&=\mumin(H^0(X,a\cdot{\bf L}),\xi_{a\cdot{\bf\phi}})\\
    &\geq \mumin(H^0(X,a\cdot{\bf L}),\xi_{a\cdot\shfL})-\lVert a\rVert_1 \delta.
\end{align*}
We are thus done by setting $S:=S_0-\delta$ and $T:=T_0-C_\scrX$.
\end{proof}
Now we finished the preparation for the proof of continuity of $\chi$-volume on $\widehat{\Div}_\Q(X)$. From now on, we further assume that $X$ is geometrically integral.
\begin{theo}\label{theo_conti_chi_vol}
Let $\overline D=(D,g)$, $\overline E_1=(E_1,h_1)$, $\dots, \overline E_r=(E_r,h_r)$ be adelic $\Q$-Cartier divisors on $X$. Then we have the following continuity:
\begin{equation}\label{eq_conti_semiample}\lim_{\substack{\lvert\epsilon_1\rvert+\cdots+\lvert\epsilon_r\rvert\rightarrow 0\\\epsilon_i\in\Q}}\vol_\chi(\overline D+\sum_{i=1}^r\epsilon_i \overline E_i)=\vol_\chi(\overline D).\end{equation}
\end{theo}
\begin{proof}
We may assume that all $D$ and $E_i$ are Cartier divisors. 
Then there exists constants $S$ and $T$ depending on $\overline D$ and $\overline E_1,\dots,\overline E_r$, such that
$$\mumin(H^0(X,n_0 D+\sum_{i=1}^r n_i E_i),\xi_{n_0g+\sum_{i=1}^r n_i h_i})\geqslant T+S\sum_{i=0}^{r}\lvert n_i\rvert$$
where $n_i\in \Z$.
For arbitrary $\epsilon_i\in\Q$, we write that $\epsilon_i=p_i/q_i$ where $p_i$ and $q_i$ are coprime integers and $q_i\geqslant 1$. Let $q=\prod_{i=1}^r q_i$, then it holds that
$$\mumin(H^0(X,mq(D+\sum_{i=1}^r \epsilon_i E_i)),\xi_{mq(g+\sum_{i=1}^r \epsilon_i h_i)})\geqslant T+Smq(1+\sum_{i=1}^r \lvert\epsilon_i\rvert)$$
for every $m\in\N$.
Therefore
$$\mumin^{\sup}(\overline D+\sum_{i=1}^r \epsilon_i \overline E_i)\geqslant S(1+\sum_{i=1}^r\lvert\epsilon_i\rvert).$$
Take $\nu$-integrable functions $\phi$ such that
$\displaystyle\int_\Omega \phi \nu(d\omega) >-S$. Denote that $\epsilon=(\epsilon_i)_{1\leq i\leq r}$ and $\lVert\epsilon\rVert_1=\sum_{i=1}^r \lvert\epsilon_i\rvert$. It holds that
\begin{equation}\label{eq_vol_chi_semiample}\vol(\overline D+\sum_{i=1}^r \epsilon_i \overline E_i+(0,(1+\lVert\epsilon\rVert_1)\phi))=\vol_\chi(\overline D+\sum_{i=1}^r \epsilon_i \overline E_i+(0,(1+\lVert\epsilon\rVert_1)\phi)).\end{equation}
Due to the continuity of $\vol(\cdot)$\citep[Theorem 6.4.24]{adelic}, (\ref{eq_conti_semiample}) can be easily derived from (\ref{eq_vol_chi_semiample}).
\end{proof}

The reason we consider the continuity of $\vol_\chi(\cdot)$ is that we can deduce the Hilbert-Samuel formula for nef adelic line bundles. Before we get into that, we introduce the arithmetic intersection theory over function fields. We say an adelic line bundle $\ovl L$ is \textit{integrable} if we can write $\ovl L=(L_1,\phi_1)-(L_2,\phi_2)$ where $L_i$ are ample, $\phi_i$ are semipositive and $(L_i,\phi_i)$ are adelic line bundles for $i=1,2$. We also say the corresponding adelic Cartier divisor is integrable.
We denote by $\widehat{\mathrm{Int}}(X)$ the set of all integrable adelic Cartier divisors. Let $\widehat{\mathrm{Int}}_\Q(X)$ be the subspace of $\widehat{\Div}_{\Q}(X)$ generated by $\widehat{\mathrm{Int}}(X)$.
We can define a multi-linear form \begin{align*}
    \widehat{\mathrm{Int}}_\Q(X)^{d+1}&\rightarrow \R,\\
    (\ovl D_1,\cdots, \ovl D_{d+1})&\mapsto \ovl D_1\cdots \ovl D_{d+1}.
\end{align*}
We have several ways of defining this intersection number over function fields. For example, consider it as a limit of geometric intersection numbers as in \citep{Faber_2009,luo2021relative}. Or we can calculate the intersection number as a resultant as in \citep{huayimoriwaki_equi}. For an integrable adelic line bundle $\ovl L$, we denote its self intersection number as $\widehat c_1(\ovl L)^{d+1}$. The normalized one $\displaystyle h_{\ovl L}(X):=\frac{\widehat c_1(\ovl L)^{d+1}}{(d+1)c_1(L)^d}$ is called the \textit{height} of $X$ with respect to $\ovl L$. If $L$ is ample and $\phi$ is semipositive, due to \citep[5.5.1]{huayimoriwaki_equi}, we have $$\widehat{c}_1(\ovl L)^{d+1}=\vol_\chi(\ovl L).$$
In the function field case, the above equation can be also proved by using the Grothendieck-Riemann-Roch theorem.
For the case that $L$ is nef, we give the following:
\begin{prop}\label{prop_hil_sam}
Let $\ovl L=(L,\phi)$ be an adelic line bundle such that $L$ is nef and $\phi$ is semipositive. Then we also have the equation $\widehat{c}_1(\ovl L)^{d+1}=\vol_\chi(\ovl L)$.
\end{prop}
\begin{proof}
Let $\ovl A=(A,\psi)$ be an adelic line bundle such that $A$ is ample and $\psi$ is semipositive. Then for any $\epsilon\in\Q_{>0}$, we have $\vol_\chi(\ovl L+\epsilon\ovl A)=\widehat{c}_1(\ovl L+\epsilon\ovl A)^{d+1}$ in the sense of adelic $\Q$-Cartier divisors. Letting $\epsilon\rightarrow 0$, we are done by Theorem \ref{theo_conti_chi_vol}.
\end{proof}

Hence we can see that in Theorem \ref{theo_equi_fun}, the term $\displaystyle\frac{\vol_\chi(\ovl L)}{(d+1)\mathrm{vol}(L)}$ is just the height of $X$ with respect to $\ovl L$.
%%%%%%%%%%%%%%%%%%%%%%%%%%%%%%%%%%%%%%%%%%%%%%%%%
\bibliographystyle{plain}
\bibliography{mybibliography}
\end{document}